\documentclass{file}

\usepackage{pat}
\usepackage[utf8]{inputenc}
\usepackage{amsthm,amsmath,amssymb}
\usepackage{mathtools}
\usepackage{thmtools}
\usepackage[table]{xcolor}
\usepackage{hyperref}
\usepackage{todonotes}
\makeatletter
\@ifpackageloaded{cleveref}{}{\usepackage[nameinlink,noabbrev]{cleveref}}
\makeatother
\usepackage[longnamesfirst,numbers,sort&compress]{natbib}

\definecolor{brightmaroon}{rgb}{0.76,0.13,0.28}
\providecommand{\defin}[1]{\emph{\textcolor{brightmaroon}{#1}}}
\newcommand{\bd}{\partial}
\newcommand{\Om}{\Omega}
\newcommand{\Aut}{\mathsf{Aut}}
\newcommand{\stab}{\mathsf{Stab}}
\newcommand{\Ind}{\operatorname{Ind}}

\makeatletter
\@ifundefined{rem}{}{}
\makeatother
\newtheorem{ques}[thm]{Question}

\title{\MakeUppercase{Hyperfiniteness of Boundary Actions via Tree Decompositions}}
\author{
Chris Karpinski\thanks{Department of Mathematics and Statistics, McGill University, Montreal, Canada}
\qquad
Bobby Miraftab\thanks{School of Computer Science, Carleton University, Ottawa, Canada.}
}
\date{}

\begin{document}
\maketitle

\begin{abstract}
We study conditions for a countable group acting on a connected locally finite hyperbolic graph to induce a hyperfinite orbit equivalence relation on the Gromov boundary of the graph in terms of tree-decompositions of the graph. 

We prove that for a connected locally finite hyperbolic graph $X$ equipped with an action of a countable group $G$, if $(T, \beta)$ is a $G$-invariant tree-decomposition of $X$ such that each bag induces a connected subgraph $X_t$ of $X$ for each $t \in V(T)$, each adhesion set is finite and such that there are only finitely many $G$-orbits of edges of $T$, then the orbit equivalence relation of $G$ acting on the Gromov boundary $\partial X$ is hyperfinite provided the orbit equivalence relation of $G$ acting on $\partial T$ is hyperfinite and the orbit equivalence relations of the bag stabilizers acting on $\partial X_t$ are all hyperfinite. 
We show that the converse also holds if $(T, \beta)$ satisfies the additional property that each adhesion set distinguishes at least two ends of $X$. 
\end{abstract}

\section{Introduction}

Hyperfiniteness is one of the fundamental regularity properties of countable
Borel equivalence relations.  A countable Borel equivalence relation is
\defin{Borel hyperfinite} if it is an increasing union of finite Borel
equivalence relations.  
In other words, it has the same Borel complexity as an
orbit relation generated by a Borel action of $\mathbb Z$,
\cite{DoughertyJacksonKechris1994,JacksonKechrisLouveau2002,KechrisMiller2004}.
A basic problem is to determine which naturally occurring group actions induce
hyperfinite orbit equivalence relations.

Boundary actions of groups acting on hyperbolic spaces form an important class
of examples.  If $G$ is a hyperbolic group, then its action
on its Gromov boundary induces a Borel hyperfinite orbit equivalence relation,
by a theorem of \citet{MarquisSabok2020}. It is worth mentioning that Naryshkin and Vaccaro
gave a shorter proof and established the stronger conclusion that this action
has finite Borel asymptotic dimension \cite{NaryshkinVaccaro2025}.  Related
hyperfiniteness results are known for several other boundary actions, see
\cite{HuangSabokShinko2020,PrzytyckiSabok2021,
Karpinski2025RelativelyHyperbolic,Oyakawa2024Acylindrical,
KunnawalkamOyakawaShinkoSpaas2024,KarpinskiOsajdaOyakawa2026,
Oyakawa2026VirtuallySpecial}.

In this paper, we study hyperfiniteness of boundary actions of countable groups  on the Gromov boundaries of connected locally finite hyperbolic graphs.  
Let $X$ be a connected locally finite hyperbolic graph and $G \leq \Aut(X)$ be countable.
The action of $G$ on $X$ induces a continuous
action on the Gromov boundary $\bd X$, and hence a countable Borel orbit
equivalence relation $E_G^{\bd X}$.  We are mainly interested in
cocompact (quasi-transitive) actions, but our main results require no such cocompactness assumption.

We note that the corresponding boundary action need not be hyperfinite for an arbitrary
transitive action, even when the graph is a regular tree.  Indeed, let $T_3$ be
the $3$-regular tree and fix an edge $vw$.  After deleting this edge, let $C$
be the component containing $w$, rooted at $w$.  Thus $C$ is the rooted binary
tree, and
\[
A:=\{\xi\in\bd T_3:
\text{the ray from $v$ to $\xi$ begins with the edge $vw$}\}
\]
is a clopen subset of $\bd T_3$ naturally identified with $\bd C$.  By
\cite[Theorem~4]{AbertVirag2005}, there exist $a,b\in\Aut(C)$ such that
$H:=\langle a,b\rangle\cong\mathbb F_2$ and every non-trivial element of $H$
fixes only finitely many vertices of $C$.  The action
$H\curvearrowright\bd C$ is therefore free because every non-trivial element of $H$ only fixes finitely many vertices, including $w$. 
In addition, it follows from \cite{Grigorchuk_2016} that the uniform Bernoulli probability measure on the boundary of a regular tree is invariant under its full automorphism group. Therefore, $H$ preserves the uniform Bernoulli
probability measure on $\partial C$.  Since $\mathbb F_2$ is nonamenable, we infer that $E_H^{\partial C}$ is not
measure-hyperfinite and hence is not Borel hyperfinite; see
\cite[Proposition~2.5(ii)]{JacksonKechrisLouveau2002}.  Extend every element of
$H$ to an automorphism of $T_3$ by making it act trivially outside $C$.  Let
$\Lambda:=C_2*C_2*C_2$ act simply transitively on $T_3$, and put
$G:=\langle H,\Lambda\rangle$.  Then $G$ is finitely generated and acts
transitively on $T_3$, while $E_H^{\partial C}$ is a non-hyperfinite Borel subequivalence relation of
${E_G^{\bd T_3}}_{|A}$.  Hence $E_G^{\bd T_3}$ is not Borel hyperfinite, since Borel sub-equivalence relations of hyperfinite Borel equivalence relations are again hyperfinite.

Our aim is to develop local-to-global criteria for hyperfiniteness of boundary actions on connected locally finite hyperbolic graphs. 
For example, consider a free product with amalgamation $G=A \ast_C B$, where $C$ is finite. Suppose that $A$ and $B$ are hyperbolic. It then follows that $G$ is hyperbolic, since $C$ is finite. Choose finite symmetric generating sets $S_A$ and $S_B$ containing
$C-\{1\}$, and put
\[
X:=\operatorname{Cay}(G,S_A\cup S_B),\qquad
X_A:=\operatorname{Cay}(A,S_A),\qquad
X_B:=\operatorname{Cay}(B,S_B).
\]


Then by \cite{MarquisSabok2020}, we know that $E_G^{\bd X}$ is Borel hyperfinite and $E_A^{\bd X_A}$ and $E_B^{\bd X_B}$ are hyperfinite.  In addition, denoting $T$ the Bass--Serre tree of the amalgam $ G=A\ast_C B$, we have that $E_G^{\partial T}$ is hyperfinite by \cite{KunnawalkamOyakawaShinkoSpaas2024}. We note that $X$ can be decomposed in a tree-like manner in terms of $X_A$ and $X_B$.

One can ask the more general question of hyperfiniteness of boundary actions on hyperbolic graphs admitting a decomposition into subgraphs arranged in a tree-like manner. In structural graph theory, this notion is called a \emph{tree-decomposition}.

A \defin{tree-decomposition} of a graph $X$ is a pair $(T,\beta)$,
where $T$ is a tree and $\beta\colon V(T)\longrightarrow\mathcal P(V(X))$
assigns to every node $t\in V(T)$ a set $\beta(t)$, called the \defin{bag at $t$}, such that:
\begin{enumerate}
\item $V(X)=\bigcup_{t\in V(T)}\beta(t)$;
\item every edge of $X$ has both endpoints in some bag $\beta(t)$;
\item for every $x\in V(X)$, the set $T_x:=\{t\in V(T):x\in\beta(t)\}$
is a connected subtree of $T$.
\end{enumerate}

We call the vertices of $T$ \defin{nodes}, in order to distinguish them
from the vertices of $X$. The \defin{adhesion set} of an edge $e=uv\in E(T)$ is $S_e:=\beta(u)\cap\beta(v)$.
The decomposition has \defin{finite adhesion} if every $S_e$ is finite. 
Let $G\leq\Aut(X)$.  The tree-decomposition $(T,\beta)$ is
\defin{$G$-invariant} if $G$ acts on $T$ by tree automorphisms and
\[
\beta(gt)=g\beta(t)
\qquad
\text{for every }g\in G\text{ and }t\in V(T).
\]
Equivalently, the assignment $t\mapsto\beta(t)$ is $G$-equivariant.
Tree-decompositions which are $G$-invariant are also called \defin{$G$-canonical}, see \cite{MiraftabRuhmann2022,MiraftabStavropoulos2021}.
For every $t\in V(T)$, the node stabilizer $G_t:=\stab_G(t)$ preserves $\beta(t)$ setwise and therefore acts on the induced subgraph of $\beta(t)$, which we denote by $X_t$.

Let $X$ be a graph and $(T, \beta)$ a tree-decomposition of $X$. For an edge $e=uv\in E(T)$, let $T_u^e$ and $T_v^e$ be the components of
$T-e$ containing $u$ and $v$, respectively, and put
\[
V_u^e:=\bigcup_{s\in V(T_u^e)}\beta(s),
\qquad
V_v^e:=\bigcup_{s\in V(T_v^e)}\beta(s).
\]

We will use the following lemma throughout the paper. It is a straightforward consequence of the axioms of a tree-decomposition.

\begin{lem}
\label{lem:tree-edge-separation}
For every $e=uv\in E(T)$, $V_u^e\cup V_v^e=V(X)$,
$V_u^e\cap V_v^e=S_e$,
and no edge of $X$ joins $V_u^e-S_e$ to $V_v^e-S_e$.  In particular, if
$S_e$ is finite, then the two sides $V_u^e$ and $V_v^e$ of $e$ are separated in $X$ by $S_e$.
\end{lem}

Tree-decompositions are among the fundamental tools of structural
and algorithmic graph theory.  The notions of tree-decomposition and
tree-width were introduced, under different names, by Halin in 1976
\cite{Halin1976,Diestel2025}.  They were reintroduced and developed
systematically by Robertson and Seymour in the Graph Minors project in
the 1980s
\cite{RobertsonSeymour1984,RobertsonSeymour1986}, and have since become
standard tools both in graph structure theory and in algorithms for
graphs of bounded tree-width.  

Now our main theorem explores local-to-global criteria for boundary actions on connected locally finite hyperbolic graphs $X$ via a tree-decomposition $(T, \beta)$. In other words, we relate hyperfiniteness of the boundary action on $\partial X$ to the hyperfiniteness of the boundary action on $\partial T$ and the boundary actions of the bag stabilizers on the boundaries of the bags. 

\medskip
\noindent\textbf{Main theorem.}
Let $X$ be a connected locally finite hyperbolic graph, let
$G\leq\Aut(X)$ be countable, and let $(T,\beta)$ be a $G$-invariant
tree-decomposition such that every bag induces a connected subgraph, the
adhesion is finite, and $G\backslash E(T)$ is finite.  Put
$X_t:=X[\beta(t)]$ and $G_t:=\stab_G(t)$.  If
$E_G^{\bd T}$ and all $E_{G_t}^{\bd X_t}$ are Borel hyperfinite, then
$E_G^{\bd X}$ is Borel hyperfinite.  For every edge $e=uv \in E(T)$, denoting $T_u^e$ (respectively, $T_v^e$) the components of $T - e$ containing $u$ (respectively, $v$), if each of the sets $V_u^e :=\bigcup_{s \in V(T_u^e)} \beta(s)$ and $V_v^e:=\bigcup_{s \in V(T_v^e)} \beta(s)$ contain a tail of every ray in any given end of $X$, then the converse also holds.
\medskip

Every end of $X$ orients the edges of $T$. Indeed, for every edge $e = uv \in E(T)$, by \Cref{lem:tree-edge-separation}, the complement of the adhesion set $S_e$ separates $X$ into the two disjoint sets $V_u^e - S_e$ and $V_v^e - S_e$. For any end $\omega$ of $X$, because $S_e$ is finite, exactly one of $V_u^e$ or $V_v^e$ contains a tail of every ray in $\omega$. The edge $e$
is oriented towards $u$ or $v$ depending on which of $V_u^e$ or $V_v^e$ contains a tail of every ray in the given end $\omega$.  The resulting
orientation has either a unique sink $t \in V(T)$ (in which case we say that $\omega$ is \defin{captured} by $t$) or points towards a unique end $\eta$ of
$T$ (in which case we say that $\omega$ is \defin{captured} by $\eta$).  Since every point in the Gromov boundary $\partial X$ of $X$ corresponds to a unique end of $X$, we obtain a decomposition
\[
\bd X
=
\bd_T X
\sqcup
\bigsqcup_{t\in V(T)}\bd_tX,
\]
where $\bd_T X$ is the set of points in $\partial X$ whose corresponding end is captured by an end of $T$, and $\bd_tX$ is the set of points in $\partial X$ whose corresponding end is captured the vertex $t \in V(T)$. We crucially use the above decomposition of $\partial X$ to relate $E_G^{\partial X}$ to $E_G^{\partial T}$ and $E_{G_t}^{\partial X_t}$ for each node $t$. We note that our main theorem applies to prove hyperfiniteness of boundary actions in situations not covered in the literature \cite{HuangSabokShinko2020,PrzytyckiSabok2021,
Karpinski2025RelativelyHyperbolic,Oyakawa2024Acylindrical,
KunnawalkamOyakawaShinkoSpaas2024,KarpinskiOsajdaOyakawa2026,
Oyakawa2026VirtuallySpecial}; see \Cref{prop:nonacylindrical-example}. 

The existence of tree-decompositions as in the main theorem for general connected locally finite graphs is not always guaranteed. However, through the framework of accessibility, we deduce the existence of such tree-decompositions for \emph{cocompact} connected locally finite hyperbolic graphs.

Accessibility theory for groups originates in the work of Wall \cite{Wall_71} and was developed by
Dunwoody \cite{Dunwoody1985}.  For cocompact graphs it was initiated by
Thomassen and Woess \cite{ThomassenWoess1993} and developed further in
\cite{HamannLehnerMiraftabRuehmann,HamannAccessibility,
CarmesinHamannMiraftab2022}. We note that cocompact locally finite hyperbolic graphs are accessible by \cite{HamannAccessibility}.

\section{Descriptive set theory and Borel equivalence relations}

All groups in this paper are countable.  If a group $G$ acts by Borel
automorphisms on a standard Borel space $Y$, we write $E_G^Y$ for the orbit
equivalence relation
\[
y\mathrel{E_G^Y}y'
\quad\Longleftrightarrow\quad
y'=gy\text{ for some }g\in G.
\]

We begin with the standard Lusin--Souslin theorem.

\begin{lem}{\rm\cite[Theorem~15.2]{Kechris}}
\label{lem:lusin-souslin}
Let $Y$ and $Z$ be standard Borel spaces, and let $f\colon Y\to Z$ be an
injective Borel map.  Then $f(Y)$ is Borel in $Z$, and
$f\colon Y\to f(Y)$ is a Borel isomorphism.
\end{lem}

We shall repeatedly use the following elementary properties. The following results are well-known (see, for instance, \cite{DJK94} for background on Borel hyperfinite equivalence relations), but we record the proof for convenience of the reader.

\begin{lem}
\label{lem:hyperfinite-heredity}
Let $E$ be a Borel hyperfinite countable Borel equivalence relation on a
standard Borel space $Y$.
\begin{enumerate}
\item If $A\subseteq Y$ is Borel, then $E_{|A}$ is Borel hyperfinite.
\item Every Borel subequivalence relation $R\subseteq E$ is Borel hyperfinite.
\end{enumerate}
\end{lem}

\begin{proof}
Write $E=\bigcup_n F_n$, where $F_1\subseteq F_2\subseteq\cdots$ are finite
Borel equivalence relations.  For the first assertion, use
$F_n\cap(A\times A)$.  For the second, use $F_n\cap R$.  In each case, the
resulting relations are increasing finite Borel equivalence relations whose
union is the required relation.
\end{proof}

\begin{lem}
\label{lem:transport}
Let $E$ and $F$ be countable Borel equivalence relations on standard Borel
spaces $Y$ and $Z$.  If a Borel isomorphism $\phi\colon Y\to Z$ satisfies $y\mathrel E y'$ if and only if $\phi(y)\mathrel F\phi(y')$, then $E$ is Borel hyperfinite if and only if $F$ is Borel hyperfinite.
\end{lem}

\begin{proof}
Transport an increasing sequence of finite Borel equivalence relations through
$\phi$.  The converse follows by applying the same argument to $\phi^{-1}$.
\end{proof}

If $Z$ is a standard Borel space, $Y \subseteq Z$ a Borel subset and $E$ a Borel equivalence relation on $Z$, then we say that $Y$ is \defin{$E$-invariant} if whenever $y \in Y$ and $zEy$, then $z \in Y$. Equivalently, $Y$ is a union of $E$-equivalence classes.

\begin{lem}
\label{lem:countable-union}
Let $Y=\bigsqcup_{i\in I}Y_i$ be a countable Borel partition, and let $E$ be a
countable Borel equivalence relation on $Y$ such that every $Y_i$ is
$E$-invariant.  If each $E_{|Y_i}$ is Borel hyperfinite, then $E$ is Borel
hyperfinite.
\end{lem}

\begin{proof}
Enumerate $I=\{i_0,i_1,\ldots\}$.  For every $j$, write
$E_{|Y_{i_j}}=\bigcup_n E_{j,n}$, where the $E_{j,n}$ are increasing finite
Borel equivalence relations.  Define $F_n$ to agree with $E_{j,n}$ on
$Y_{i_j}$ for $j\leq n$ and be equality on $Y_{i_j}$ for $j>n$.  Then
$(F_n)$ is increasing, every $F_n$ is finite and Borel, and $E=\bigcup_nF_n$.
\end{proof}

Given $H \leq G$ countable groups and a Borel action $H \curvearrowright Z$ on a standard Borel space $Z$, we can always enlarge this action to an action of $G$ on a standard Borel space, via the \emph{induced $G$-space} construction.

Let $H\leq G$ and let $Z$ be a standard Borel $H$-space.  The
\defin{induced $G$-space} is
\[
\Ind_H^G(Z):=(G\times Z)/{\sim},
\qquad
(gh,z)\sim(g,hz),
 \quad \text{ for every $g \in G, h \in H$ and $z \in Z$}\]
with $G$ acting by left multiplication on the first coordinate.

\begin{lem}
\label{lem:induced-hyperfinite}
If $E_H^Z$ is Borel hyperfinite, then $E_G^{\Ind_H^G(Z)}$ is Borel
hyperfinite.
\end{lem}

\begin{proof}
Choose a set $C\subseteq G$ of representatives for the left cosets $G/H$.
Then $\Ind_H^G(Z)$ is Borel isomorphic to $C\times Z$, and
\[
(c,z)\mathrel{E_G^{\Ind_H^G(Z)}}(c',z')
\quad\Longleftrightarrow\quad
z\mathrel{E_H^Z}z'
 \text{ for each $c,c' \in C$ and $z,z' \in Z$} \]
Write $E_H^Z=\bigcup_n F_n$, with $(F_n)$ increasing and finite Borel.  Choose
finite sets $C_1\subseteq C_2\subseteq\cdots$ with union $C$.  On $C\times Z$,
let $R_n$ relate $(c,z)$ and $(c',z')$ if
$c,c'\in C_n$ and $z\mathrel{F_n}z'$.  Each $R_n$ is a finite Borel
equivalence relation, the sequence is increasing, and its union is the induced
orbit relation.
\end{proof}

\begin{lem}
\label{lem:saturation-induced}
Let $G$ act by Borel automorphisms on a standard Borel space $Y$, let $A$ be a countable
$G$-set, and suppose
\[
Y=\bigsqcup_{a\in A}Y_a
\]A
is a Borel partition into nonempty sets satisfying $gY_a=Y_{ga}$.  Fix
$a_0\in A$, put $H:=\stab_G(a_0)$ and $Z:=Y_{a_0}$, and set
\[
B:=\bigsqcup_{a\in Ga_0}Y_a=GZ.
\]
Then $B\cong_G\Ind_H^G(Z)$ as Borel $G$-spaces.
\end{lem}

\begin{proof}
We first define 
$\Phi\colon\Ind_H^G(Z)\longrightarrow B,
\Phi([g,z])=gz.$
This is well-defined, $G$-equivariant, and onto.  If $gz=g'z'$ with
$z,z'\in Z$, apply $g^{-1}$ to this equality.  The resulting common point
belongs both to $Y_{a_0}$ and to $Y_{g^{-1}g'a_0}$.  Since the partition
pieces are nonempty and disjoint, $g^{-1}g'a_0=a_0$.  Thus $g^{-1}g'\in H$, and $[g,z]=[g',z']$.  Hence $\Phi$
is injective.  Using a left-coset transversal, $\Phi$ is a countable union of
Borel maps.  It is therefore Borel, and \Cref{lem:lusin-souslin} completes the
proof.
\end{proof}

We now isolate the exact descriptive-set-theoretic input used later.

\begin{thm}
\label{thm:abstract-transfer}
Let $G$ act by Borel automorphisms on a standard Borel space $Y$.  Suppose
\[
Y=Y_0\sqcup\bigsqcup_{a\in A}Y_a
\]
is a countable Borel partition such that $Y_0$ is $G$-invariant, $A$ is a
countable $G$-set, every $Y_a$ is nonempty, and $gY_a=Y_{ga}$.  Let $W$ be a
standard Borel $G$-space, and suppose that there is a $G$-equivariant Borel
injection
$\phi\colon Y_0\longrightarrow W$.
Choose one representative $a_i$ from each $G$-orbit in $A$, and put
$H_i:=\stab_G(a_i)$. If $E_G^W$ and every $E_{H_i}^{Y_{a_i}}$ are Borel hyperfinite, then
$E_G^Y$ is Borel hyperfinite.
Conversely, if $E_G^Y$ is Borel hyperfinite, then every
$E_{H_i}^{Y_{a_i}}$ is Borel hyperfinite.  If $\phi$ is onto $W$, then $E_G^W$ is also Borel
hyperfinite, and consequently
$E_G^Y$ is Borel hyperfinite if and only if 
$E_G^W$ is Borel hyperfinite
and $E_{H_i}^{Y_{a_i}}$ is Borel hyperfinite for every $i$.

\end{thm}

\begin{proof}
By \Cref{lem:lusin-souslin}, $\phi(Y_0)$ is a $G$-invariant Borel subset of
$W$, and $\phi$ is a Borel isomorphism from $Y_0$ onto its image.  Equivariance
and injectivity show that it is an isomorphism of the corresponding orbit
relations.  Hence $E_G^{Y_0}$ is Borel hyperfinite by
\Cref{lem:hyperfinite-heredity,lem:transport}.

For each $i$, put
\[
B_i:=\bigsqcup_{a\in Ga_i}Y_a.
\]
By \Cref{lem:saturation-induced},
$B_i\cong_G\Ind_{H_i}^G(Y_{a_i})$.  Thus $E_G^{B_i}$ is Borel hyperfinite by
\Cref{lem:induced-hyperfinite,lem:transport}.  Finally,
\[
Y=Y_0\sqcup\bigsqcup_iB_i
\]
is a countable $G$-invariant Borel partition, so
\Cref{lem:countable-union} gives the forward implication.

Now suppose that $E_G^Y$ is Borel hyperfinite.  For every $i$,
$E_{H_i}^{Y_{a_i}}$ is a Borel subequivalence relation of
${E_G^Y}_{|Y_{a_i}}$, and hence is Borel hyperfinite by
\Cref{lem:hyperfinite-heredity}.  If $\phi$ is onto, then since $Y_0$ is Borel and
$G$-invariant, we have that $E_G^{Y_0}$ is Borel hyperfinite; transporting through $\phi$
gives hyperfiniteness of $E_G^W$ by \Cref{lem:transport}.
\end{proof}

\section{Ends, boundaries, and tree-decompositions}

All graphs are simple and endowed with their path metrics.  Let $X$ be a
connected locally finite graph.  A \defin{ray} is a one-way infinite simple
path.  Two rays are equivalent if, after deleting any finite vertex set, they
have tails in the same component.  An equivalence class of rays is a
\defin{end}; the set of ends of $X$ is denoted by $\Om(X)$.  If
$S\subseteq V(X)$ is finite and $C$ is a component of $X-S$, then
\[
\Om(C):=\{\omega\in\Om(X):
\text{every ray in $\omega$ has a tail in $C$}\}
\]
is a basic open set in the end topology (see \cite[Section 8.6]{Diestel} for more on the end topology on the space of ends of a graph).  The end space of a connected locally
finite graph is compact (see \cite[Proposition 8.6.1]{Diestel}).  If $T$ is a tree, we write $\bd T:=\Om(T)$.  When
$T$ is countable, $\bd T$ is a standard Borel space (it may be realized as a Borel subset of $\{0,1\}^{E(T)}$).

Now suppose that $X$ is hyperbolic.  Its \defin{Gromov boundary} $\bd X$ consists of
equivalence classes of geodesic rays, where two rays are equivalent if they
have finite Hausdorff distance.  Every geodesic ray determines a end, so
there is a natural map $\epsilon_X\colon\bd X\longrightarrow\Om(X)$.

\begin{lem}
\label{lem:epsilon} \cite[Lemma~I.8.28(2), p.~145, and
Exercise~III.H.3.8, p.~430]{BH99}
If $X$ is connected, locally finite, and hyperbolic, then
$\epsilon_X\colon\bd X\to\Om(X)$ is onto and continuous.
\end{lem}

Let $(T,\beta)$ be a finite-adhesion tree-decomposition of a connected graph
$X$.  An end $\omega\in\Om(X)$ orients every edge $e=uv$ of $T$ toward
the unique side containing a tail of every ray in $\omega$.  This is
well-defined by \Cref{lem:tree-edge-separation}.

A vertex $t\in V(T)$ \defin{captures} $\omega$ if $t$ is a sink in this
orientation.  An end $\eta\in\bd T$ \defin{captures} $\omega$ if every edge of
$T$ is oriented toward the component containing a tail of a ray in $\eta$. The following is well-known, but we provide a short proof for convenience of the reader.

\begin{lem}
\label{lem:end-capture}
Every end of $X$ is captured by exactly one element of
$V(T)\sqcup\bd T$.
\end{lem}

\begin{proof}
At every vertex of $T$ there is at most one outgoing edge.  Indeed, suppose
that distinct edges $tu_1$ and $tu_2$ were both oriented away from $t$, and let
$C_i$ be the component of $T-tu_i$ containing $u_i$.  The two corresponding
sides intersect only in a subset of
$(\beta(t)\cap\beta(u_1))\cap(\beta(t)\cap\beta(u_2))$, which is finite.  A
ray cannot have a tail in both sides.

If some vertex has no outgoing edge, it is a sink.  There cannot be two sinks,
since the path between them would contain a vertex with two outgoing edges.

Suppose instead that every vertex has exactly one outgoing edge. Following outgoing edges from any vertex produces a ray.  
Any two such directed rays eventually meet; otherwise the path joining them would contain a vertex with two outgoing edges.  They therefore determine one end of $T$, and every edge is oriented toward that end.  
Uniqueness is immediate.
\end{proof}

Suppose now that $X$ is hyperbolic.  
We define 
$$\bd_T X
:=
\{\xi\in\bd X:
\epsilon_X(\xi)\text{ is captured by an end of }T\},$$
and, for $t\in V(T)$, we define 
$\bd_tX
:=
\{\xi\in\bd X:
\epsilon_X(\xi)\text{ is captured by }t\}$.
Then we have 
\begin{equation}
\label{eq:capture-partition}
\bd X
=
\bd_T X
\sqcup
\bigsqcup_{t\in V(T)}\bd_tX.
\end{equation}
For $\xi\in\bd_T X$, let $\kappa(\xi)$ be the unique end of $T$ capturing
$\epsilon_X(\xi)$.  This defines the \defin{capture map}
$\kappa\colon\bd_T X\longrightarrow\bd T$.

\begin{lem}
\label{lem:capture-equivariant}
Let $(T,\beta)$ be a $G$-invariant finite-adhesion tree-decomposition.  
Then $g\bd_tX=\bd_{gt}X$ for $g\in G,\ t\in V(T)$,
$\bd_T X$ is $G$-invariant, and the capture map $\kappa$ is
$G$-equivariant.
\end{lem}

\begin{proof}
For every end $\omega$, the orientation associated with $g\omega$ is the
$g$-translate of the orientation associated with $\omega$.  Also
$\epsilon_X(g\xi)=g\epsilon_X(\xi)$.  The assertions follow directly from the
definitions.
\end{proof}

\begin{lem}
\label{lem:capture-continuous}
Let $X$ be a connected locally finite hyperbolic graph and let $(T,\beta)$ be a finite-adhesion tree-decomposition.  Then
$\kappa\colon\bd_T X\longrightarrow\bd T$ is continuous, where $\bd_T X$ has the subspace topology inherited from
$\bd X$.
\end{lem}

\begin{proof}
Fix an oriented edge $\vec e=(u,v)$ of $T$, and let $T_v^e$ be the component
of $T-e$ containing $v$.  The set
\[
\mathcal O_{\vec e}
:=
\{\eta\in\bd T:
\text{$e$ points toward $T_v^e$ in the orientation toward $\eta$}\}
\]
is clopen in $\bd T$, and finite intersections of such sets over all oriented edges $e$ form a basis for the topology on $\Omega(X)$.
Let
\[
U_{\vec e}
:=
\{\omega\in\Om(X):
\text{every ray in $\omega$ has a tail in }V_v^e\}.
\]
By \Cref{lem:tree-edge-separation}, after deleting the finite set $S_e$, each
of $V_u^e-S_e$ and $V_v^e-S_e$ is a union of components of $X-S_e$.  Hence
$U_{\vec e}$ and its complement are open in $\Om(X)$, so $U_{\vec e}$ is clopen.  
Moreover, we have $\kappa^{-1}(\mathcal O_{\vec e})
=
\bd_T X\cap\epsilon_X^{-1}(U_{\vec e})$.
This is open in $\bd_T X$ by \Cref{lem:epsilon}.  
Therefore $\kappa$ is continuous.
\end{proof}

\section{Geometric control of the boundary pieces}

We now introduce the kinds of tree-decompositions we will work with. These are the tree-decompositions with connected bags and such that the adhesion sets have uniformly bounded cardinalities and diameters. 
We use the conventions $\operatorname{diam}(\varnothing)=0$ and
$\sup\varnothing=0$.

\begin{defn}
\label{def:bounded}
Let $X$ be a connected locally finite graph. A finite-adhesion
\newline tree-decomposition $(T,\beta)$ is \defin{bounded} if:
\begin{enumerate}
\item $T$ is countable;
\item every induced bag graph $X_t:=X[\beta(t)]$ is connected;
\item for every $t\in V(T)$,
$D_t:=
\sup_{ut \in E(T)}
\operatorname{diam}_{X_t}(\beta(t)\cap\beta(u))
<\infty$;
\item there are $M,D<\infty$ such that, for every $e\in E(T)$,
$|S_e|\leq M$ and $\operatorname{diam}_X(S_e)\leq D$.
\end{enumerate}
\end{defn}

\begin{lem}
\label{lem:bag-boundary}
Let $X$ be a connected locally finite hyperbolic graph, let $(T,\beta)$ be a
finite-adhesion tree-decomposition, and let $t\in V(T)$ be given.  Assume that $X_t$ is connected and that
\[
D_t=
\sup_{ut \in E(T)}
\operatorname{diam}_{X_t}(\beta(t)\cap\beta(u))
<\infty.
\]
Then the inclusion $\iota_t\colon X_t\hookrightarrow X$ is a quasi-isometric embedding.  
In particular, $X_t$ is hyperbolic. Furthermore, $\iota_t$ induces a
homeomorphism $\bd\iota_t\colon\bd X_t\longrightarrow\bd_tX$. 
\end{lem}

\begin{proof}
Let $C$ be a component of $T-t$, and let $u$ be the neighbour of $t$ contained
in $C$.  Then
\begin{equation}
\label{eq:bag-side-intersection}
\beta(t)\cap\bigcup_{s\in V(C)}\beta(s)
=
\beta(t)\cap\beta(u).
\end{equation}
Indeed, if $x$ belongs to the left-hand side, then $T_x$ contains $t$ and a vertex of $C$, and hence contains $u$.
Let $P=v_0v_1\cdots v_k$ be a path in $X$ whose endpoints lie in
$\beta(t)$ and whose internal vertices lie outside $\beta(t)$.  Every internal
vertex $x$ has $T_x$ contained in a unique component of $T-t$.  Consecutive
internal vertices determine the same component, since an edge of $X$ is
contained in a bag.  Thus all internal vertices of $P$ lie on one side,
corresponding to a component $C$ of $T-t$ and its neighbour $u$ of $t$.
A bag containing the edge $v_0v_1$ lies in $C$; since $T_{v_0}$ also contains
$t$, its connectedness forces $u\in T_{v_0}$.  Hence
$v_0\in\beta(t)\cap\beta(u)$, and the same argument applies to $v_k$.
Therefore the endpoints of every such path belong to one incident
adhesion set.

Now let $x,y\in\beta(t)$, and let $\gamma$ be an $X$-geodesic from $x$ to $y$.  
Replace every maximal subpath of $\gamma$ outside $\beta(t)$ by a path inside $X_t$ of length at most $D_t$.  
The resulting walk contains an $X_t$-path from $x$ to $y$ of length at most
$(D_t+1)|\gamma|$.  Therefore
\[
d_X(x,y)
\leq d_{X_t}(x,y)
\leq(D_t+1)d_X(x,y),
\]
so $\iota_t$ is a quasi-isometric embedding.
The same argument shows that $X_t$ is $D_t$-quasiconvex in $X$.  Indeed, a maximal subpath of an $X$-geodesic outside $\beta(t)$ has endpoints $a,b$ in one incident adhesion set, and hence has length $d_X(a,b)\leq d_{X_t}(a,b)\leq D_t$.
Thus $X_t$, with its intrinsic path metric, is hyperbolic and the inclusion induces the asserted injective boundary map.
Let $\xi\in\bd X_t$, represented by an $X_t$-geodesic ray $\rho$.  
As a path in $X$, $\rho$ is a quasi-geodesic ray.  For every component $C$ of $T-t$ with neighbour $u$ of $t$, equation \eqref{eq:bag-side-intersection} gives
\[
\rho\cap\bigcup_{s\in V(C)}\beta(s)
\subseteq
\beta(t)\cap\beta(u),
\]
which is finite.  
Any geodesic ray representing $\bd\iota_t(\xi)$ lies at
finite Hausdorff distance from $\rho$, and therefore determines the same end.  
That end orients every edge incident with $t$ toward $t$.
Therefore $\bd\iota_t(\xi)\in\bd_tX$.

Conversely, let $\xi\in\bd_tX$, and let $r$ be an $X$-geodesic ray representing $\xi$.  
The ray $r$ meets $\beta(t)$ infinitely often.  
Otherwise a tail of $r$ would avoid $\beta(t)$; consecutive vertices of that tail would then lie on the same side of $T-t$, so the entire tail would lie in one such side, contradicting capture by $t$.

Choose vertices $p_0,p_1,p_2,\ldots$ of $r\cap\beta(t)$ tending to infinity along $r$, and let $\alpha_n$ be an $X_t$-geodesic from $p_0$ to $p_n$.  
The paths $\alpha_n$, viewed in $X$, are uniform quasi-geodesics.  
By the Morse lemma (see \cite[Theorem III.H.1.7]{BH99}), they remain within a uniform distance of the corresponding subsegments of $r$.  
Since $X_t$ is locally finite and
$d_{X_t}(p_0,p_n)\to\infty$, a subsequence converges to an $X_t$-geodesic ray
$\alpha$ from $p_0$.  A standard diagonal argument, using the two-sided
Hausdorff bound supplied by the Morse lemma, shows that the image of $\alpha$
in $X$ has finite Hausdorff distance from $r$.  Hence
$\bd\iota_t([\alpha])=\xi$.  Thus the image is exactly
$\bd_tX$.

Finally, $\bd X_t$ is compact and $\bd X$ is Hausdorff, so the continuous
bijection from $\bd X_t$ onto its image is a homeomorphism.
\end{proof}

We next prove that the capture map is injective, which is a consequence of the above bounds on the cardinality and diameter of the adhesion sets. 

\begin{lem}
\label{lem:capture-injective}
Let $X$ be a connected locally finite hyperbolic graph, and let $(T,\beta)$ be a finite-adhesion tree-decomposition.  
Let there be $M,D<\infty$ such that $|S_e|\leq M$ and $\operatorname{diam}_X(S_e)\leq D$ for every $e\in E(T)$.  Then
$\kappa\colon\bd_T X\longrightarrow\bd T$
is injective.
\end{lem}

\begin{proof}
Suppose that $\kappa(\xi)=\kappa(\zeta)=\eta$, and let
$t_0t_1t_2\cdots$ be a ray representing $\eta$.  Put
$S_n:=\beta(t_n)\cap\beta(t_{n+1})$,
let $C_n$ be the component of $T-t_nt_{n+1}$ containing
$t_{n+1},t_{n+2},\ldots$, and set
\[
V_n:=\bigcup_{s\in V(C_n)}\beta(s).
\]
Next we define $P_\eta
:=
\{x\in V(X):t_n\in T_x\text{ for all sufficiently large }n\}$.
Every finite subset of $P_\eta$ is contained in $S_n$ for all sufficiently large $n$.  
Hence $|P_\eta|\leq M$.
For a fixed vertex $x$, the set of indices $n$ for which $t_n\in T_x$ is an interval, because $T_x$ is connected.  It follows that if $x\notin P_\eta$, then $x$ lies in only finitely many $S_n$.  
Consequently, for every finite $F\subseteq V(X)$,
\begin{equation}
\label{eq:adhesion-escape}
S_n\cap F\subseteq P_\eta
\qquad\text{for all sufficiently large }n.
\end{equation}
Moreover, if $x\notin P_\eta$, then $x\notin V_n$ for all sufficiently large
$n$.  Indeed, if $T_x$ met arbitrarily deep tail components $C_n$, its
connectedness would force it to contain all sufficiently large $t_n$.

Let $r_\xi$ and $r_\zeta$ be geodesic rays from a common base vertex $o$
representing $\xi$ and $\zeta$.  Choose vertices $x\in r_\xi$ and
$y\in r_\zeta$ after the last intersections of these rays with the finite set
$P_\eta$.  For all sufficiently large $n$, the vertices $x,y$ lie outside
$V_n$, while both rays have tails in $V_n$.  By
\Cref{lem:tree-edge-separation}, there are therefore vertices
$p_n\in r_\xi\cap S_n$ and $q_n\in r_\zeta\cap S_n$
after $x$ and $y$, respectively.  In particular,
$p_n,q_n\notin P_\eta$.

Fix $R<\infty$.  The ball $B_X(o,R)$ is finite.  By
\eqref{eq:adhesion-escape}, for all sufficiently large $n$,
$S_n\cap B_X(o,R)\subseteq P_\eta$.
Thus $p_n,q_n\notin B_X(o,R)$ for all sufficiently large $n$.  Hence
\[
d_X(o,p_n)\longrightarrow\infty,
\qquad
d_X(o,q_n)\longrightarrow\infty.
\]
On the other hand, $p_n,q_n\in S_n$, so $d_X(p_n,q_n)\leq D$.  Therefore
$(p_n\mid q_n)_o\to\infty$, and the two geodesic rays determine the same point
of $\bd X$.  Hence $\xi=\zeta$.
\end{proof}

To prove surjectivity of $\kappa$, we will need the following property of our tree-decomposition. 

\begin{defn}
\label{def:end-essential}
A finite-adhesion tree-decomposition $(T,\beta)$ of $X$ is
\defin{end-essential} if, for every edge $e=uv\in E(T)$, each of the two sides
$V_u^e$ and $V_v^e$ contains a tail of every ray in an end of $X$.  
\end{defn}

\begin{lem}
\label{lem:capture-surjective}
Let $X$ be a connected locally finite hyperbolic graph, and let $(T,\beta)$ be
an end-essential finite-adhesion tree-decomposition.  Then
$
\kappa\colon\bd_T X\longrightarrow\bd T$
is surjective.
\end{lem}

\begin{proof}
Fix $\eta\in\bd T$, represented by a ray $t_0t_1t_2\cdots$.  For every $n$,
let $C_n$ be the component of $T-t_nt_{n+1}$ containing
$t_{n+1},t_{n+2},\ldots$, and define
\[
A_n
:=
\left\{
\omega\in\Om(X):
\text{every ray in $\omega$ has a tail in }
\bigcup_{s\in V(C_n)}\beta(s)
\right\}.
\]
Each $A_n$ is clopen in $\Om(X)$, by the same finite-separator argument used
in \Cref{lem:capture-continuous}.  End-essentiality gives $A_n\neq\varnothing$,
and
\[
A_0\supseteq A_1\supseteq A_2\supseteq\cdots.
\]
Since $\Om(X)$ is compact, choose
$\omega\in\bigcap_nA_n$.

We claim that $\eta$ captures $\omega$.  Let $e$ be any edge of $T$, and let
$C(e,\eta)$ be the component of $T-e$ containing a tail of the ray $\eta$.  For all sufficiently large $n$, $C_n$ is contained in
$C(e,\eta)$.  Since $\omega\in A_n$, the edge $e$ is oriented toward
$C(e,\eta)$.  This holds for every $e$, proving the claim.
It follows from \Cref{lem:epsilon} that, the end $\omega$ contains a geodesic ray and hence
is $\epsilon_X(\xi)$ for some $\xi\in\bd X$.  Then
$\xi\in\bd_T X$ and $\kappa(\xi)=\eta$.
\end{proof}

Combining the preceding lemmas gives the required Borel compatibility.

\begin{prop}
\label{prop:boundary-compatibility}
Let $X$ be a connected locally finite hyperbolic graph and let $(T,\beta)$ be
a bounded tree-decomposition.  Then:
\begin{enumerate}
\item for every $t\in V(T)$, the inclusion $X_t\hookrightarrow X$ induces a Borel isomorphism $\bd X_t\longrightarrow\bd_tX$;
\item the partition \eqref{eq:capture-partition} is a countable Borel
partition;
\item the capture map $\kappa\colon\bd_T X\to\bd T$ is a continuous Borel
injection, its image is Borel in $\bd T$, and it is a Borel isomorphism onto
that image.
\end{enumerate}
If the decomposition is also end-essential, then $\kappa$ is a Borel
isomorphism from $\bd_T X$ onto all of $\bd T$.

If the decomposition is $G$-invariant, all the maps in the statement are
equivariant for the corresponding actions of $G$ and $G_t$.
\end{prop}

\begin{proof}
The first assertion is \Cref{lem:bag-boundary}.  In particular, every
$\bd_tX$ is Borel by \Cref{lem:lusin-souslin}.  Since $T$ is countable, the
union of the bag pieces is Borel, and hence its complement $\bd_T X$ is Borel.
The map $\kappa$ is continuous by \Cref{lem:capture-continuous} and injective
by \Cref{lem:capture-injective}.  Its image is Borel and its inverse on the
image is Borel by \Cref{lem:lusin-souslin}.  If the decomposition is
end-essential, surjectivity follows from \Cref{lem:capture-surjective}.
Equivariance follows from \Cref{lem:capture-equivariant} and from equivariance
of the inclusions $X_t\hookrightarrow X$.
\end{proof}

The next observation explains the role of finite edge-orbits.

\begin{lem}
\label{lem:finite-orbit-control}
Let $X$ be a connected locally finite graph, let $G\leq\Aut(X)$ be countable, and let $(T,\beta)$ be a $G$-invariant finite-adhesion tree-decomposition.
Assume that every $X_t$ is connected and that $G\backslash E(T)$ is finite.
Then $(T,\beta)$ is bounded.  
\end{lem}

\begin{proof}
If $T$ has no edge, it has one vertex and the conclusion is immediate.
Otherwise, choose representatives $e_j=a_jb_j$, $1\leq j\leq m$, for the
finitely many $G$-orbits on $E(T)$, and put
$S_j:=\beta(a_j)\cap\beta(b_j)$.  Since each $S_j$ is finite and both incident
bag graphs are connected, the quantities $M:=\max_j|S_j|$
and
\[
D':=
\max_j
\max\left\{
\operatorname{diam}_{X_{a_j}}(S_j),
\operatorname{diam}_{X_{b_j}}(S_j)
\right\}
\]
are finite.  
By $G$-invariance, for every edge $tu\in E(T)$,
$|\beta(t)\cap\beta(u)|\leq M$, $\operatorname{diam}_{X_t}(\beta(t)\cap\beta(u))\leq D'$.

Thus every local constant $D_t$ is at most $D'$, and the diameter of
every adhesion set is also at most $D'$.  This proves the metric conditions (3) and (4) in
\Cref{def:bounded}.

Since $G$ is countable and $E(T)$ is a finite union of $G$-orbits, $E(T)$ is
countable, and hence so is $T$.  Every vertex of a nontrivial tree is an
endpoint of an edge, so the endpoints of the finitely many representative
edges meet every vertex orbit.  Therefore $G\backslash V(T)$ is finite.
\end{proof}

\section{Combination theorems}

We first state the exact boundary-level criterion for hyperfiniteness in terms of the tree-decomposition. 

\begin{thm}
\label{thm:boundary-transfer}
Let $X$ be a connected locally finite hyperbolic graph, let
$G\leq\Aut(X)$ be countable, and let $(T,\beta)$ be a countable $G$-invariant
finite-adhesion tree-decomposition.  Put $X_t:=X[\beta(t)]$ and
$G_t:=\stab_G(t)$.  Assume that:
\begin{enumerate}
\item every $X_t$ is hyperbolic, and the inclusion $X_t\hookrightarrow X$
induces a $G_t$-equivariant Borel isomorphism $
\bd X_t\longrightarrow\bd_tX$;
\item the capture map $\kappa\colon\bd_T X\longrightarrow\bd T$
is a $G$-equivariant Borel injection.
\end{enumerate}
If $E_G^{\bd T}$ and every $E_{G_t}^{\bd X_t}$ are Borel hyperfinite, then
$E_G^{\bd X}$ is Borel hyperfinite.

Conversely, if $E_G^{\bd X}$ is Borel hyperfinite, then every
$E_{G_t}^{\bd X_t}$ is Borel hyperfinite. If, moreover, $\kappa$ is onto $\bd T$, then $E_G^{\bd T}$ is Borel
hyperfinite as well.
\end{thm}

\begin{proof}
By the first assumption and \Cref{lem:lusin-souslin}, every bag piece $\bd_tX$ is Borel.  Since $T$ is
countable, \eqref{eq:capture-partition} is a countable Borel partition.  By
\Cref{lem:capture-equivariant}, the bag pieces are equivariantly permuted and
$\bd_T X$ is invariant.
Next we invoke \Cref{thm:abstract-transfer} with
$Y=\bd X, Y_0=\bd_T X, Y_t=\bd_tX, W=\bd T, \phi=\kappa,$
omitting the empty bag pieces.  The hypothesis that $E_{G_t}^{Y_t}$ is hyperfinite follows from
the given hypothesis on $\bd X_t$ by transport through the boundary map (\Cref{lem:transport}).  This
proves the forward implication.  The hyperfiniteness of $E_{G_t}^{\partial X_t}$ follows from the converse
part of \Cref{thm:abstract-transfer}, without surjectivity of $\kappa$.  If
$\kappa$ is onto, the same theorem also gives hyperfiniteness of $E_G^{\partial T}$, again
transporting between $\bd X_t$ and $\bd_tX$ using \Cref{lem:transport}.
\end{proof}

As a direct corollary of \Cref{thm:boundary-transfer}, we obtain an equivalence of hyperfiniteness of $E_G^{\bd X}$ and hyperfiniteness of $E_G^{\bd T}$ and $E_{G_t}^{\bd X_t}$ for each $t \in V(T)$ under the assumptions of \Cref{thm:boundary-transfer} and surjectivity of $\kappa$.

\begin{cor}
    \label{cor:hyperfiniteness_of_X_iff_hyperfiniteness_of_T_and_bags}
    Let $X$ be a connected locally finite hyperbolic graph, let
$G\leq\Aut(X)$ be countable, and let $(T,\beta)$ be a countable $G$-invariant
finite-adhesion tree-decomposition.  Put $X_t:=X[\beta(t)]$ and
$G_t:=\stab_G(t)$. Assume that:
\begin{enumerate}
\item every $X_t$ is hyperbolic, and the inclusion $X_t\hookrightarrow X$
induces a $G_t$-equivariant Borel isomorphism $
\bd X_t\longrightarrow\bd_tX$;
\item the capture map $\kappa\colon\bd_T X\longrightarrow\bd T$
is a $G$-equivariant Borel bijection. 
\end{enumerate}
Then $E_G^{\bd X}$ is Borel hyperfinite if and only if $E_G^{\bd T}$ is Borel hyperfinite and $E_{G_t}^{\bd X_t}$ is Borel hyperfinite for every $t\in V(T)$.
\end{cor}

We can now replace the boundary-level assumptions by the concrete metric
conditions of \Cref{def:bounded}.

\begin{cor}
\label{thm:controlled-forward}
Let $X$ be a connected locally finite hyperbolic graph, let
$G\leq\Aut(X)$ be countable, and let $(T,\beta)$ be a $G$-invariant
bounded tree-decomposition.  Put $X_t:=X[\beta(t)]$ and
$G_t:=\stab_G(t)$.  If $E_G^{\bd T}$ and every
$E_{G_t}^{\bd X_t}$ are Borel hyperfinite, then $E_G^{\bd X}$ is Borel
hyperfinite.
\end{cor}

\begin{proof}
The hypotheses of the forward part of \Cref{thm:boundary-transfer} follow from
\Cref{prop:boundary-compatibility}.
\end{proof}

\begin{cor}
\label{thm:controlled-equivalence}
Under the hypotheses of \Cref{thm:controlled-forward}, assume in addition that
$(T,\beta)$ is end-essential.  Then $E_G^{\bd X}$ is Borel hyperfinite if and only if $ E_G^{\bd T}$ is Borel hyperfinite and $E_{G_t}^{\bd X_t}$ is Borel hyperfinite for every $t\in V(T)$.

\end{cor}

\begin{proof}
By \Cref{prop:boundary-compatibility}, the capture map is now a
$G$-equivariant Borel isomorphism from $\bd_T X$ onto $\bd T$.  Apply the
equivalence part of \Cref{thm:boundary-transfer}.
\end{proof}

We are now ready to prove our main theorem.

\begin{thm}
\label{thm:main}
Let $X$ be a connected locally finite hyperbolic graph, let
$G\leq\Aut(X)$ be countable, and let $(T,\beta)$ be a $G$-invariant
tree-decomposition satisfying:
\begin{enumerate}
\item every bag $\beta(t)$ induces a connected subgraph $X_t$ of $X$;
\item $(T,\beta)$ has finite adhesion;
\item $G\backslash E(T)$ is finite.
\end{enumerate}
For $t\in V(T)$, put $G_t:=\stab_G(t)$.
If $E_G^{\bd T}$ and every $E_{G_t}^{\bd X_t}$ are Borel hyperfinite, then $E_G^{\bd X}$ is Borel hyperfinite.  Conversely, hyperfiniteness of $E_G^{\bd X}$ always implies hyperfiniteness of every $E_{G_t}^{\bd X_t}$.
In addition, if $(T, \beta)$ is end-essential, then
$E_G^{\bd X}$  is Borel hyperfinite if and only if $E_G^{\bd T}$ is Borel hyperfinite and $E_{G_t}^{\bd X_t}$ is Borel hyperfinite for every $t\in V(T)$.

\end{thm}

\begin{proof}
By \Cref{lem:finite-orbit-control}, the decomposition is bounded.
The forward implication follows from \Cref{thm:controlled-forward}.  The hyperfiniteness of $E_{G_t}^{\partial X_t}$ follows from \Cref{prop:boundary-compatibility} and the converse part
of \Cref{thm:boundary-transfer}.
The additional hypothesis says precisely that the decomposition is
end-essential, so the full equivalence follows from
\Cref{thm:controlled-equivalence}.
\end{proof}


Recall that an action of a group $G$ on a metric space $X$ is \defin{acylindrical} if for every $\varepsilon > 0$, there exists $R > 0$ and $N \in \N$ such that for all $x,y \in X$ with $d(x,y) > R$, the set 

$$ \{g \in G : d(x,gx) \leq \varepsilon \text{ and } d(y, gy) \leq \varepsilon\}$$

has cardinality at most $N$.




Our main result recovers the previously known cases while also yielding new ones. 
Indeed, suppose that a group $G$ acts properly and cocompactly on a locally finite hyperbolic graph $X$. 
By the Milnor--\v{S}varc lemma \cite[Proposition~I.8.19]{BH99}, the group $G$ is hyperbolic, the orbit map $G\to X$ is a quasi-isometry, and hence the Gromov boundary $\partial G$ is $G$-equivariantly homeomorphic to $\partial X$. It therefore follows from \cite{MarquisSabok2020} that $E_G^{\partial X}$ is hyperfinite. 
Likewise, if $X$ is a tree and the action $G\curvearrowright X$ is acylindrical, then the hyperfiniteness of $E_G^{\partial X}$ follows from \cite{KunnawalkamOyakawaShinkoSpaas2024}. 

The example below satisfies neither of these sets of hypotheses and therefore lies beyond the scope of the previously known hyperfiniteness results.

\begin{exa}
\label{prop:nonacylindrical-example}
There exists a finitely generated non-hyperbolic group $G$ acting on a
connected locally finite hyperbolic graph $X$ having a $G$-invariant tree-decomposition $(T, \beta)$ such that the action
$G\curvearrowright X$ is not acylindrical, the kernel of the
action on $T$ acts nontrivially on $\bd X$, and
$E_G^{\bd X}$ is Borel hyperfinite.
\end{exa}

\begin{proof}
Let $F_2=\langle a,b\rangle$ and let $T_0:=\operatorname{Cay}(F_2,\{a,b\})$, which is a regular tree.
We choose a connected locally finite hyperbolic graph $Y$ with a
distinguished vertex $o_Y$; for example, $Y$ may be a Cayley graph of a
closed hyperbolic surface group.
Let $R$ be the rooted binary tree, whose vertices are the finite binary
words.  
Let $u\in\Aut(R)$ be the binary odometer, recursively defined by
$u(\varnothing)=\varnothing,u(0w)=1w,u(1w)=0u(w)$.
Then $u$ has infinite order.  
Put $K:=\langle u\rangle\cong\mathbb Z$.
Identify the root of $R$ with $o_Y$, obtaining $Z:=Y\vee R$,
and extend the action of $K$ to $Z$ by letting $K$ fix $Y$ pointwise.
For every $q\in F_2$, let $Z_q$ be a copy of $Z$, with wedge point $o_q$.
For every edge $qq'\in E(T_0)$, add an edge joining $o_q$ to $o_{q'}$.
Denote the resulting graph by $X$.  The group
$G:=F_2\times K$ acts on $X$ by $(h,u^n)\cdot z_q=(u^nz)_{hq}$.

The graph $X$ is connected and locally finite.  It is also hyperbolic.
Indeed, $Z$ is hyperbolic, and each copy $Z_q$ meets the remainder of
$X$ only at $o_q$.  A geodesic joining points in distinct copies is the
concatenation of geodesics to the corresponding wedge points, the
unique path between those wedge points in $T_0$, and geodesics inside
the terminal copies.  Hence every geodesic triangle is assembled from
a tripod in $T_0$ and geodesic triangles inside copies of $Z$, giving a
uniform hyperbolicity constant.

Let $T$ be the barycentric subdivision of $T_0$.  For an original
vertex $q\in V(T_0)$, define $\beta(q):=V(Z_q)$.
For the subdivision vertex $m_e$ corresponding to an edge
$e=qq'\in E(T_0)$, define $\beta(m_e):=\{o_q,o_{q'}\}$.
Then $(T,\beta)$ is a $G$-invariant tree-decomposition of $X$, where $G$ acts on $T$ via the natural action of $Q$ on $T$ and letting $K$ act trivially on $T$.  
Every bag induces a connected subgraph, every adhesion set is a singleton,
and $G\backslash E(T)$ is finite.
For every original node $q$, one has $\stab_G(q)=K\cong\mathbb Z$.
Thus the orbit equivalence relation induced by
$\stab_G(q)$ on $\bd Z_q$ is Borel hyperfinite.  The bags corresponding
to subdivision vertices are finite and therefore have empty boundary.
Moreover, the action of $G$ on $\bd T$ factors through the standard
action of $Q=\mathbb F_2$ on the boundary of its Cayley tree, and hence
$E_G^{\bd T}$ is Borel hyperfinite.  Theorem~\ref{thm:main} now gives
that $E_G^{\bd X}$ is Borel hyperfinite.

Finally, $G$ is not hyperbolic, since it contains
$\langle a\rangle\times K\cong\mathbb Z^2$.  The action on $X$ is not
acylindrical because the infinite subgroup $K$ fixes every wedge point
$o_q$, and these wedge points contain pairs at arbitrarily large
distance.  On the other hand, the kernel $K$ of the action of $G$ on $T$ acts nontrivially on $\partial X$ because it acts nontrivially on the copy of
$\bd R$ contained in $\bd Z_q\subseteq\bd X$.  
\end{proof}

\section{Applications}

In this section, we discuss applications of our results to accessible graphs.
A connected locally finite graph $X$ is \defin{vertex-accessible} if there is
$k\in\mathbb N$ such that every two distinct ends of $X$ can be separated by at most $k$ vertices, i.e.\ there exists a set $F \subseteq V(X)$ with $\vert F \vert \leq k$ such that there are rays in the two given ends with tails lying in different components of $X - F$.
The following is the only place where cocompactness is used to obtain a
decomposition.

\begin{lem}
\label{lem:accessibility-decomposition}
Let $X$ be a connected locally finite vertex-accessible graph, and let a
countable group $G\leq\Aut(X)$ act cocompactly on $X$.  Then $X$ admits
a $G$-invariant tree-decomposition $(T,\beta)$ such that:
\begin{enumerate}
\item every bag induces a connected subgraph of $X$;
\item the decomposition has finite adhesion;
\item $G\backslash E(T)$ is finite;
\item the decomposition is end-essential.
\end{enumerate}
\end{lem}

\begin{proof}
If $X$ has at most one end, take the one-bag decomposition.  Suppose that $X$
has at least two ends.  By the accessibility decomposition theorem
\cite[Theorem~6.4]{HamannLehnerMiraftabRuehmann}, there is a $G$-invariant finite-adhesion
tree-decomposition that distinguishes all ends
efficiently (where by \emph{efficiently}, we mean that the adhesion sets have the minimum cardinality needed to distinguish the ends) and has finitely many $G$-orbits on $E(T)$.  In the construction,
every edge whose adhesion set does not efficiently separate a pair of ends is
contracted.  Thus the decomposition may be chosen so that every remaining
edge distinguishes two ends, which is assertion~(4).

By \cite[Proposition~4.1 and Lemma~4.2]
{HamannLehnerMiraftabRuehmann}, the bags may be enlarged equivariantly so that
they induce connected subgraphs.  The tree is unchanged, the new adhesion
sets remain finite, and each original bag is contained in the corresponding
enlarged bag.  Consequently,
an end living on one side of an edge before the enlargement still lives on
that side afterward.  Thus the two ends distinguished by any edge remain on
opposite sides.  Hence all four assertions hold.
\end{proof}

\begin{cor}
\label{cor:hyperbolic-decomposition}
Let $X$ be a connected locally finite hyperbolic graph, and let a countable
group $G\leq\Aut(X)$ act cocompactly on $X$.  Then $X$ admits a
decomposition satisfying all the hypotheses of the equivalence part of
\Cref{thm:main}.
\end{cor}

\begin{proof}
By \cite[Theorem~4.3]{HamannAccessibility}, every locally finite
cocompact hyperbolic graph is vertex-accessible.  Apply
\Cref{lem:accessibility-decomposition}.
\end{proof}

\begin{cor}
\label{cor:quasitransitive-reduction}
Let $X$ and $G$ be as in \Cref{cor:hyperbolic-decomposition}.  For any
decomposition as in \Cref{lem:accessibility-decomposition}, with
$X_t:=X[\beta(t)]$ and $G_t:=\stab_G(t)$, one has that $E_G^{\bd X}$ is Borel hyperfinite if and only if $E_G^{\bd T}$ is Borel hyperfinite and 
$E_{G_t}^{\bd X_t}$ is Borel hyperfinite for every $t\in V(T)$.
\end{cor}

\section{Further questions}

\citet{NaryshkinVaccaro2025} strengthened the Marquis--Sabok theorem \cite{MarquisSabok2020} by
proving finite Borel asymptotic dimension for hyperbolic-group boundary
actions.  Recall that the Borel asymptotic dimension of an action is defined
using a proper right-invariant metric on the acting group and is independent of
the choice of such a metric; see \cite[Lemma~2.2]{Borel_asym_dim_2023}.  It is
natural to ask whether the decomposition method above also preserves this
stronger invariant.

\begin{ques}
Let $X$ be a connected locally finite hyperbolic graph, let
$G\leq\Aut(X)$ be countable, and let $(T,\beta)$ be a $G$-invariant
bounded end-essential tree-decomposition.  Put
$X_t:=X[\beta(t)]$ and $G_t:=\stab_G(t)$.  Suppose
$\operatorname{asdim}_{\mathrm B}(G\curvearrowright\bd T)<\infty$ and $\sup_{t\in V(T)}
\operatorname{asdim}_{\mathrm B}(G_t\curvearrowright\bd X_t)<\infty$.
Must $\operatorname{asdim}_{\mathrm B}(G\curvearrowright\bd X)<\infty?$
\end{ques}

The induction step already raises a separate problem.

\begin{ques}
Let $H\leq G$ be countable groups and let $Z$ be a standard Borel $H$-space.
Suppose $\operatorname{asdim}_{\mathrm B}(H\curvearrowright Z)\leq d$ and $\operatorname{asdim}_{\mathrm B}(G\curvearrowright G/H)\leq q$.
Under what additional hypotheses does one have
\[
\operatorname{asdim}_{\mathrm B}
\bigl(G\curvearrowright\Ind_H^G(Z)\bigr)
\leq d+q?
\]
In particular, when does finite Borel asymptotic dimension pass from the
$H$-action on $Z$ and the $G$-action on $G/H$ to the induced action?
\end{ques}

\section*{Acknowledgments}

We are grateful to Joseph MacManus, Koichi Oyakawa, Antoine Poulin and Marcin Sabok for carefully reading previous drafts of the paper and offering many helpful comments. This research was conducted while the first author was visiting Carleton University.  This research was supported by NSERC.

\section*{AI statement}

We have used ChatGPT to help write and structure the paper. However, we emphasize that the main mathematical content and ideas of the paper are due to the authors and not generated by AI.

\bibliographystyle{plainurlnat}
\bibliography{MPF}

\end{document}